\documentclass[11pt,twoside]{article}
\usepackage{latexsym}
\usepackage{amssymb,amsbsy,amsmath,amsfonts,amssymb,amscd}
\usepackage{graphicx}
\usepackage{hyperref}
\usepackage{xcolor}
\newcommand{\commentout}[1]{}

\newcommand{\R}{\mathbb{R}}

\newcommand {\e}  {\varepsilon}

\newcommand {\Chi} {{\bf \raise 2pt \hbox{$\chi$}} }
\newcommand {\f}   {\frac}
\newcommand {\p}   {\partial}
\newcommand {\ep}  {\epsilon}
\newcommand{\dis}{\displaystyle}
\newcommand {\proof} {\noindent {\bf Proof}. }
\newcommand{\beq}{\begin{equation}}
\newcommand{\eeq}{\end{equation}}
\newcommand{\bea} {\begin{array}{rl}}
\newcommand{\eea} {\end{array}}
\newcommand{\bepa}{\left\{ \begin{array}{l}}
\newcommand{\eepa} {\end{array}\right.}
\newtheorem{theorem}{Theorem}[section]
\newtheorem{lemma}[theorem]{Lemma}
\newtheorem{definition}[theorem]{Definition}

\newtheorem{proposition}[theorem]{Proposition}

\newcommand {\no}{\noindent}
\newcommand{\qed}{{ \hfill
                       {\unskip\kern 6pt\penalty 500 \raise -2pt\hbox{\vrule\vbox to 6pt{\hrule width 6pt
                       \vfill\hrule}\vrule} \par}   }}
\title{ {Viscosity solutions for junctions: well posedness and stability }}

 \author{
Pierre-Louis Lions$^{1}$ and Panagiotis Souganidis$^{2,3}$}

\date{\today}

\begin{document}

\maketitle
\pagestyle{plain}
\pagenumbering{arabic}

\begin{abstract}
\no 
\smallskip

\no We introduce a notion of state-constraint viscosity solutions for one dimensional ``junction''-type problems for Hamilton-Jacobi equations with non convex coercive Hamiltonians and  study its well-posedness and stability properties. We show that viscosity approximations either select the state-constraint solution or have a unique limit. We also introduce 
another type of approximation by fattening the domain. We also make connections with existing results for convex equations and discuss extensions to time dependent and/or  multi-dimensional problems.

\end{abstract}

\\
\noindent {\bf AMS Class. Numbers.} 35F21, 49L25, 35B51, 49L20.
\bigskip


\section{The problem and the notion of solution.}

\no We introduce a notion of state-constraint viscosity solutions for one dimensional junction-type problems for non convex Hamilton-Jacobi equations and  study its well-posedness (comparison principle  and existence). We also investigate the stability properties of small diffusion approximations satisfying a Kirchoff property at the junction. We show that such approximations either converge to the state-constraint solution or have a unique limit. We also introduce a new type of approximations by ``fattening'' the junction, which under some assumptions on the behavior of the Hamiltonian's at the junction, also yield the state-constraint.  We also present a new and very simple proof for the uniqueness of the junction solutions introduced for quasi-convex problems by Imbert and Monneau \cite{IM}. Finally, we discuss 
extensions to time dependent and/or  multi-dimensional problems.

\smallskip

\no For simplicity and due to the space limitation we concentrate here on  one-dimensional time independent problems. Our results, however, extend with some additional technicalities, to time dependent  as well as multi-dimensional ``stratified''  problems. Proofs  as well extensions to multi-dimensional  problems will appear in \cite{LS1}.
\smallskip

\no We emphasize that our results  do not require any convexity conditions on the Hamiltonians contrary to all the previous literature that is based on the control theoretical interpretation of the problem and, hence, require convexity. Among the long list of references on this topic with convex Hamiltonians, in addition to \cite{IM}, we refer to Barles and Briani and Chasseigne\cite{BBC1,  BBC2}, Barles and Chasseigne \cite{BC}, Bressan and Hong \cite{BH} and Imbert and Nguen \cite{IN}.
\smallskip

\no We consider a $K$-junction problem in the domain  $I:=\bigcup_{i=1}^{K} I_i$ and junction $\{0\}$, where, for $i=1,\ldots,K$,  $I_i :=(-a_i, 0)$ and $a_i \in [-\infty, 0)$. 
We work with functions $u \in C({\bar I}; \R)$ and,  for $x=(x_1,\ldots,x_K) \in \bar I$, we write $u_i(x_i)=u(0, \ldots, x_i,\dots, 0)$;  when possible, to simplify the writing,  we drop the subscript on $u_i$ and simply write $u(x_i)$. 
 We also use the notation $u_{x_i}$ and $u_{x_i x_i}$ for the first and second derivatives of $u_i$ in $x_i$.   Finally, to avoid unnecessarily long statements, 
we do not repeat, unless needed, that $i=1,\dots,K$.
 \smallskip

For the Hamiltonians $H_i \in C(\R \times  I; \R)$ 
we assume that, for each  $i$,  
\begin{equation}\label{takis1}
H_i  \text{ is coercive, that is} \ H_i (p_i,x_i) \to \infty \  \text{as \ $|p_i| \to \infty$ \ uniformly on $\bar I_i$.}
\end{equation} 

Next we present the definitions of the state-constraint sub- and super-solutions.

\begin{definition}\label{takis0} (i) $u \in C({\bar I}; \R)$  is a state-constraint sub-solution to the junction problem if  
\begin{equation}\label{takis3.1}
u_i+H_i(u_{x_i}, x_i)\leq 0 \ \text{in} \ I_i \ \text{ for each $i$.} 
\end{equation}
(ii)~ $u \in C({\bar I}; \R)$ is a state-constraint super-solution to the junction problem if 
\begin{equation}\label{takis3.2}
 u_i+H_i(u_{x_i}, x_i)\geq 0 \ \text{in} \ I_i \  \text{for each $i$},  
\end{equation}
and
\begin{equation}\label{takis3.3}
 u(0) + \underset {1\leq i \leq K}\max  H_i(u_{x_i},0) \geq 0.
\end{equation}

(iii)~$u \in C({\bar I}; \R)$ is a solution if it is both sub-and super-solution.
\end{definition}
The super-solution inequality at the junction is interpreted in the viscosity sense, that is  if, for $\phi \in C^1(I)\cap C^{0,1}(\bar I)$, $u-\phi$ has a (strict local) minimum at $x=0$, then $ u(0) + \underset {1\leq i \leq K}\max H_i(\phi_{x_i}(0),0) \geq 0.$
\smallskip

The definition of the state constrain solution says that $u$ is a solution if it is a viscosity solution in $I$ and a constrained super-solution in $\bar I_i$ for at least one $i$. 
\smallskip

We remark that, for the sake of brevity, we are not precise about the boundary conditions at the end points $a_i$, which may be of any kind (Dirichlet, Neumann or state-constraint) that yields comparison for solutions in each $I_i$.
\smallskip

We also note that, without much difficulty, it is possible to study with more than one  junctions, since, as it will become apparent from the proofs below, the ``influence'' of the each junction is ``local''.
\smallskip

Finally, we denote by $u^{sc,i}\in C({\bar I}_i)$ the unique constraint-solution to 
$w+H_i(w_{x_i}, x_i)=0$ on ${\bar I}_i.$

\smallskip

\section{The main results}


We begin with the well posedness of the state-constraint solution of the junction problem.
\begin{theorem}\label{pl1} Assume \eqref{takis1}. 

(i)~If $v,u \in C(\bar I)$ are respectively sub-and super-solutions to the junction problem, then $v\leq u$ on $\bar I$. 
\vskip.075in
(ii)~There exists a unique state-constraint solution $\hat u$ of the junction problem.
\vskip.075in
(iii)~$\hat u(0)=\underset{1\leq i \leq K} \min u^{sc,i}(0),$ where $u^{sc,i}$ is the state-constraint solution to 
$w+H_i(w_{x_i}, x_i)=0$ on ${\bar I}_i.$
\end{theorem}
Since it is classical in the theory of viscosity solutions that the comparison principle yields via Perron's existence method, here we will not discuss this any further.
\smallskip

The second result is about the stability properties of ``viscous'' approximations to the junction problem. We begin with the formulation
and the well-posedness of solutions to  second-order uniformly elliptic equations on junctions satisfying a possibly nonlinear Neumann (Kirchoff-type) condition.
\smallskip

We assume that  
the continuous functions   $F_i:=F(X_i,p_i,u_i,x_i) \ \text{and}  \ G:=G(p_1, \ldots, p_K, u)$ are (uniformly with respect to all the other arguments)
\begin{equation}\label{takis4}
\begin{cases}
F_i  \text{  strictly decreasing in $X_i$, nonincreasing in $u_i$, and coercive in $p_i$;}\\[1mm]
\text{$G$ strictly increasing with respect to the $p_i$'s and nonincreasing with respect to $u$,}
\end{cases}
\end{equation}
and consider the problem 
\begin{equation}\label{pl1.5}
\begin{cases}
F_i(u_{x_i x_i}, u_{x_i},x_i,u_i, x_i)=0  \  \text{in} \  I_i \ \text{for each $i$}\\[1mm]
G(u_{x_1},\ldots, u_{x_K},u)= 0 \ \text{ on} \  \{0\}.
\end{cases}
\end{equation}
\begin{theorem}\label{pl2} Assume \eqref{takis4}. Then \eqref{pl1.5}  has a unique solution $\hat u\in C^2( I)\cap C^{1,1}(\bar I).$
\end{theorem}
The meaning of the Neumann condition at the junction is that $G$ quantifies the ``amount'' of the diffusion that goes into each direction as well as stays at $0$.
\smallskip
 
We consider next, for each $\epsilon >0$, the problem 
\begin{equation}\label{takis5}
\begin{cases}
-\ep u^\ep _{x_i x_i} + u^\ep_i +H_i(u^\ep_{x_i},x_i)=0  \  \text{in} \  I_i,\\[1mm]
\sum_{i=1}^K u^\ep_{x_i}= 0 \ \text{ on} \  \{0\},
\end{cases}
\end{equation}
which, in view of Theorem \ref{pl2}, has a unique solution $u^\ep \in C^2( I)\cap C^{1,1}(\bar I),$ that, in addition, is bounded in $ C^{0,1}(\bar I)$ with a bound independent of the $\ep$; the uniform in $\ep$ bound is an easy consequence of the assumed coercivity of the Hamiltonian's.  
\smallskip

We remark that the particular choice of the Neumann condition plays no role in the sequel and results similar to the ones stated below will also hold true for other, even nonlinear, conditions at the junction.
\smallskip

We are interested in the behavior, as $\ep \to 0$, of the $u^\ep$'s and, in particular, in the existence of a unique limit and its relationship to the constraint solution of the first-order junction problem.
\begin{theorem}\label{pl3}
Assume \eqref{takis1}. Then $u:=\underset{\ep \to 0}\lim u^\ep$ exists and either $u=\hat u$ or $u(0) < \hat u(0)$, $u_{x_i}(0^-)$ exists for all $i$'s and \  $\sum_{i=1}^Ku_{x_i}(0^-)= 0.$
\end{theorem}
A consequence of Theorem \ref{pl3} is that, in principle, the junction problem has a unique state-constraint solution and a possible continuum of solutions obtained as limits of problems like \eqref{takis5} with other type of possibly degenerate second order terms and different Neumann conditions.
\smallskip

Under some additional assumptions it is possible to show that we always have  $\hat u=\underset{\ep \to 0}\lim u^\ep$. Indeed suppose that, for each $i$, 
\begin{equation}\label{takis2}
H_i  \ \text{has no flat parts and finitely many minima at} \ p^0_{i,1} \leq \ldots \leq p^0_{i,K_i};
\end{equation}
note  that the assumption that $H_i$ has no flat parts can be easily removed by a density argument, while, at the expense of some technicalities, it is not necessary to assume that there are only  finitely minima. 
\begin{theorem}\label{pl4}  Assume \eqref{takis1}, \eqref{takis2} and $\sum_{i=1}^K   p^0_{i,K_i}\leq 0$. Then $\hat u=\underset{\ep \to 0}\lim u^\ep$.
\end{theorem}
A particular case that \eqref{takis2} holds is when the $H_i$'s are quasi-convex and coercive.  Then, for each $i$, there exists single minimum point at $p^0_i,$
and the condition above reduces to $\sum_{i=1}^K   p^0_i \leq 0.$ On the other hand, if  $\sum_{i=1}^K   p^0_i > 0$, we have  examples showing that $\hat u > \underset{\ep \to 0}\lim u^\ep.$

\section{Sketch of proofs}
\label{sec:proofs}

The proof of Theorem \ref{pl2} is standard so we omit it and  
we present  the one  of Theorem \ref{pl1}.
\smallskip

\begin{proof} It follows from  \eqref{takis1} that $v$ is Lipschitz continuous. In view of the comments in the previous section  about the boundary conditions at the $a_i$'s, here we assume that $v(0) -u(0) = \max_{\bar I} (u-v) >0$ and we obtain a contradiction.  
\smallskip

To conclude we adapt the argument introduced in Soner \cite{So} to study state-constraint problems and we consider, for each $i$,
$\ep >0$ and some $\delta=O(\ep),$ a maximum point $({\bar x}_i, {\bar y}_i)\in \ {\bar I}_i \times {\bar I}_i $ (over  ${\bar I}_i\times {\bar I}_i $) of   
$(x_i,y_i) \to v(x_i)-u(y_i) - \frac{1}{2\ep}( {\bar x}_i -{\bar y}_i +\delta)^2.$ 
\smallskip

It follows that, as $\ep \to 0$,  ${\bar x}_i, {\bar y}_i \to 0$, and  the role of the $\delta$ above is to guarantee that, for all $i$,  $ {\bar x}_i <0$ even if ${\bar y}_i=0.$ 
\smallskip

If, for some $j$,   ${\bar y}_j<0$, we find, using the uniqueness arguments for state-constraint viscosity solutions in ${\bar I}_j$, a contradiction to $v(0) -u(0)>0$.   
\smallskip

It follows that we must have ${\bar y}_i=0$ for all $i=1,\dots, K$, that is,  $y \to v(y) + \frac{1}{2\ep} \underset {i} \sum (\bar x_i - y_i + \delta)^2$ has a minimum at $0$. Since $v$ is a super-solution,  \eqref{takis3.3} yields $ v(0) + \underset {1\leq i \leq K}\max H_i(\frac{{\bar x}_i +\delta}{\ep}, 0)\geq 0$ and, hence, for some $j$,  
$ v_j(0) +  H_j(\frac{{\bar x}_j +\delta}{\ep},0)\geq 0$.
 
\smallskip

On the other hand, since $ {\bar x}_j <0$, we also have 
$u_j({\bar x}_j) + H_j(\frac{{\bar x}_j +\delta}{\ep},{\bar x}_j )\leq 0.$ 

\smallskip

Combining the last two inequalities we find, after letting $\ep \to 0$,  that we must have $u(0)=u_j(0) \leq v_j(0)=v(0)$, which again contradicts the assumption.
\smallskip

The existence of a unique solution $\hat u$ follows from the comparison and Perron's method.
\smallskip

For the  third claim first we observe that, since $\hat u$ is a viscosity sub-solution in each $I_i$, the comparison of state-constraint solutions yields that, for each $i$,  $\hat u\leq u^{sc,i}$ on ${\bar I}_i$, and, hence, $\hat u(0) \leq \underset{1\leq i\leq K} \min u^{sc,i}(0)$.  
\smallskip 

For  the equality, we need to show that, for some $j$, $u^{sc,j}(0)\leq \hat u(0)$. This follows by repeating the proof of the comparison above. 

\hskip6.5in  $\blacksquare$

\end{proof} 
\medskip

To study the limiting behavior of the $u^\ep$'s, we  investigate in detail the properties of solutions to the Dirichlet problem in each of the intervals $I_i$. For notational simplicity we omit next the dependence on $i$ and we consider, for each $c\in \R$, the boundary value problem
\begin{equation}\label{takis6}
u_c + H(u_{c,x},x)= 0 \ \text {in}  \  I:=(-a,0) \ \text{and} \   u(0)=c,
\end{equation}
and we denote by $u^{sc}$ the solution of the corresponding state constraint problem in $I$;  note that, since the real issue is the behavior near $0$, again we do not specify any  boundary condition at  $a$, which can be either Dirichlet or Neumann or state constrain so that  \eqref{takis6} is well defined. Finally, as we already mentioned earlier,  we use \eqref{takis2} is only to avoid technicalities. 
\begin{proposition}\label{pl5} 
Assume that $H$ satisfies \eqref{takis1} and  \eqref{takis2}. Then, for every $c<u^{sc}(0),$ \eqref{takis6} has a unique solution $u_c \in C^{0,1}(\bar I)$. Moreover,  $u_{c,x}(0^-)$ exists and
$u_c(0^-) + H(u_{c,x}(0^-),0)= 0$. In addition, both $u_c(0^-)$ and $u_{c,x}(0^-)$ are nondecreasing in $c$, and $u_{c,x}(0^-)$ belongs to the decreasing part of $H$.
\end{proposition}
\begin{proof} 
The existence of solutions to \eqref{takis6} is immediate from Perron's method, since, for any $\lambda >0$,  $u^{sc}- \lambda$ is a sub-solution, while the coercivity of the $H$ easily yields a super-solution. The Lipschitzcontinuity of the solution is an immediate consequence of the coercivity of $H$. The existence of  $u_{c,x}(0^-)$ and the fact the equation is satisfied at $0$ follow either along the lines of Jensen and Souganidis \cite{JS}, which studied the detailed differentiability properties of viscosity solutions in one dimension, or a technical lemma  stated without proof after the end of the ongoing one. The claimed monotonicity of $u_c(0^-)$ follows from the comparison principle, while the monotonicity of $u_{c,x}(0^-)$ is a consequence of the fact that, for any $c\neq c'$, the maximum of $u_c-u_{c'}$ is attained at $x=0$. The last assertion  results from the nondecreasing properties  of $u_c(0^-)$ and $u_{c,x}(0^-)$ and  the fact that $u_c(0^-) + H(u_{c,x}(0^-),0)= 0$.

\hskip6.5in  $\blacksquare$
\end{proof} 

\smallskip

The technical lemma that can be used in the above proof in place  of \cite{JS} is stated next without a proof. 

\begin{lemma}\label{pl6}   Assume that $u \in C^{0,1}(\bar I)$ solves  $u+H(u_x,x)\leq 0$ (resp. $u+H(u_x,x)\geq 0$)  in $I$ and let $\bar p:= \underset{x\to 0^-} \limsup \frac{u(x)-u(0)}{x}$ and $\underline p:=\underset{x\to 0^-} \liminf \frac{u(x)-u(0)}{x}.$ Then $u(0)+ H(\bar p, 0) \leq 0$ (resp.  $u(0)+ H(\underline p, 0) \geq 0$.) 
\end{lemma}


We state next without a proof a well known fact which characterizes the possible limits of the uniform in $\ep$ Lipschitzcontinuous solutions $u^\ep$  to  \eqref{takis5}.

\begin{lemma}\label{pl7} Assume \eqref{takis1}. 
Any subsequential  
 limit $u$ of the $u^\ep$ is a viscosity sub-solution to 
\begin{equation}
\begin{cases}\label{takis11}
u + H_i(u_{x_i}, x_i)\leq 0 \ \text{in} \ I_i \ \text{ for each $i$},\\[1mm]   
\min [\sum_{i=1}^K u_{x_i},  u(0) + \underset{1\leq i \leq K}\min H_i(u_{x_i}, 0)] \leq 0 \ \text{at $x=0$,}  
\end{cases}
\end{equation}
and a viscosity super-solution to
\begin{equation}
\begin{cases}\label{takis12}
u + H_i(u_{x_i}, x_i)\geq 0 \ \text{in} \ I_i \ \text{ for each $i$},\\[1mm]
\max [\sum_{i=1}^K u_{x_i},  u(0) + \underset{1\leq i \leq K}\max H_i(u_{x_i}, 0)] \geq 0 \ \text{at $x=0$.}  
\end{cases}
\end{equation}
\end{lemma}
Recall that the inequalities at $x=0$ must be interpreted in the viscosity sense. For example, if, for some $\phi \in C^{0,1}(\bar I)$, $u-\phi$ has a maximum at $0$, then
$\min [\sum_{i=1}^d \phi_{x_i}(0^-),  u(0) + \underset{1\leq \ i \leq K}\min H_i(\phi_{x_i}(0), 0)] \leq 0.$
 \smallskip
 
Proposition \ref{pl5} below refines the behavior of any $u$ satisfying \eqref{takis11} and \eqref{takis12}.  The proof of Theorem \ref{pl3} is then immediate. 

\begin{proposition}\label{pl8}
Assume \eqref{takis1} and \eqref{takis2}. 

(i) If $u$ is  continuous solution  to  \eqref{takis11} and \eqref{takis12} and  $u(0)<\hat u(0)$, then  $\sum_{i=1}^d u_{x_i}(0^-) =0.$ 

(ii) The problem \eqref{takis11} and \eqref{takis12} has at most one  solution on $u\in C^{0,1}(\bar I)$ such that  $u(0)<\hat u(0)$.
\end{proposition}
\begin{proof}(i)~Proposition \ref{pl5} yields that, for each $i$,  the $u_{x_i}(0^-)$'s exist  and belong to the decreasing part of the $H_i $and $u(0)+ H_i(u_{x_i}(0^-),0)=0.$ 
It follows that there exists some small $\lambda >0$ such that $u(0)+ H_i(u_{x_i}(0^-) + \lambda,0)<0$ and  $u(0)+ H_i(u_{x_i}(0^-) -\lambda,0)>0.$
\smallskip

Choose $\phi^{\pm} \in C^{0,1}(\bar I)$ be such that $\phi^{\pm}_{x_i}(0^-)=u_{x_i}(0^-) \pm \lambda$. It follows that $0$ is a local max and min of $u-\phi^{-}$ and  $u-\phi^{+}$ respectively. Then \eqref{takis11} and \eqref{takis12} and the choice of $\phi^{\pm}$ yield the inequalities\\

$\begin{cases}
\min \Big[\sum_{i=1}^K \phi^-_{x_i}(0^-),  u(0) + \underset{1\leq i \leq K}\min H_i(\phi^-_{x_i}(0^-), 0)\Big]=\\[1mm]
\qquad \qquad  \min \Big[\sum_{i=1}^K u_{x_i}(0^-) -\lambda K,  u(0) + \underset{1\leq i \leq K}\min H_i(u_{x_i}(0^-) -\lambda, 0)\Big]\leq 0, 
\end{cases}$
\smallskip

and

\smallskip

$\begin{cases}
\max \Big[\sum_{i=1}^K \phi^+_{x_i}(0^-),  u(0) + \underset{1\leq i  \leq K}\max H_i(\phi^+_{x_i}(0^-), 0)\Big]=\\[1mm]
\qquad \qquad \max \Big [\sum_{i=1}^K u_{x_i}(0^-) +\lambda K,  u(0) + \underset{1\leq i \leq K}\max H_i(u_{x_i}(0^-)+\lambda , 0)\Big] \geq 0. \end{cases}$
\smallskip

It follows from the choice of $\lambda$ that  $\sum_{i=1}^K u_{x_i}(0^-) -\lambda \leq 0 \leq \sum_{i=1}^K u_{x_i}(0^-) +\lambda K$, and, hence, letting $\lambda \to 0 $ yields the claim. 

\smallskip
(ii) If $u,v$ are two continuous solutions to \eqref{takis11} and \eqref{takis12},   the Kirchoff condition established above  implies that, for some small $\delta>$, $u(x)-v(x) -\delta \sum_{i=1}^d x_i$ cannot have a maximum at $0$. The claim then follows from standard viscosity solutions arguments.

\hskip6.5in  $\blacksquare$
\end{proof}
\smallskip

Theorem \ref{pl4} is now immediate from the first claim in Proposition 3.4. 

\section{Some observations }
We present  another way to approximate the constrained solution of the junction based on ``fattening'' $\bar I$.  To simplify the notation we assume that $K=2$. 
\smallskip

For $\ep>0$, let $I_\ep$ be an open neighborhood of  ${\bar I}$ in $\R^2$ of size $\ep$, that is  ${\bar I} \subset I_\ep$ and $\text{diam} I_\ep \leq \ep$,  consider the coercive Hamiltonian $H:\R^{2} \times \R^2 \to \R$ and the state-constraint problem
\begin{equation}\label{takis14}
\begin{cases}
u^\ep + H(Du^\ep, x)\leq 0 \ \text{in}  \ I_\ep,\\[1mm]
u^\ep + H(Du^\ep, x)\geq 0  \ \text{on}  \ \bar {I_\ep},
 \end{cases} 
\end{equation}
where $Dv:=(v_{x_1}, v_{x_2})$ and $x:=(x_1,x_2).$ The coercivity of $H$ yields Lipschitz bounds so that, along subsequences,  $u^\ep \to u.$  
\smallskip

Define $H_1(p_1,x_1):=\underset{p_2\in \R} \min H(p_1,p_2, x_1,0)$ and $H_2(p_2,x_2):=\underset{p_1\in \R} \min H(p_1,p_2, 0,  x_2),$
\begin{theorem}\label{pl9} Any limit $u$ of the solutions $u^\ep$ to \eqref{takis14} is a  solution to 
$u+H_1(u_{x_1},x_1)=0 \ \text{in} \  I_1 \ \text{and} \  u+H_2(u_{x_2},x_2)=0 \ \text{in} \  I_2$, and if, for some $\phi \in C^1(\R^2)$, $u-\phi$ has local minimum at $0$,
then 
 $u+H( \phi_{x_1}(0),  \phi_{x_2}(0), 0) \geq 0.$
\end{theorem}
\begin{proof} The proof of the second claim is immediate. Here we concentrate on the first part and, since the arguments are similar, we take  $i=1$.
\smallskip

For some $\phi \in C^1(I_1)$, let  ${\bar x}_1\in I_1$ be a local minimum  of $u(x_1, 0) -\phi(x_1)$. It is immediate  that, for all $p_2 \in \R$,  $u^\ep(x_1,x_2)- \phi (x_1) -p_2x_2$ has a minimum at $({\bar x}_1^\ep,  {\bar x}_2^\ep)$ and, as $\ep \to 0$, ${\bar x}_1^\ep \to x_1$ and ${\bar x}_2^\ep \to 0$. It follows from \eqref{takis14} that 
$ u({\bar x}_1,0) + H(\phi({\bar x}_1, p_2, {\bar x}_1, 0) \geq 0$, and, since $p_2$ is arbitrary,  $u({\bar x}_1,0) + H_1(\phi({\bar x}_1, {\bar x}_1 ) \geq 0.$
\smallskip

The sub-solution property follows from the fact that 
$u^\ep + H_1(u^\ep_{x_1}, x_1) \leq u^\ep + H(u^\ep_{x_1}, u^\ep_{x_1}, x_1, 0).$
\end{proof}
\medskip

An immediate consequence of Theorem \ref{pl9} is the following proposition.
\begin{proposition}\label{pl10} If $H(p_1,p_2, x_1,x_2)=\max (H_1(p_1,x_1), H_2(p_2,x_2))$, then the  $\underset{\ep \to 0}\lim u^\ep$ exists and is the state-constraint solution to the junction problem.
\end{proposition}
In general, however,  it is not true that $H(p_1,p_2, x_1,x_2)=\max (H_1(p_1,x_1), H_2(p_2,x_2))$. Indeed, if 
$H(p_1,p_2)= p_1^2 + 10 p_2^2,$ then  $H_1(p_1)=p_1^2 \  \text{ and } \   H_2(p_2)=10 p_2^2 \  \text{ and }   p_1^2 + 10 p_2^2 \ne \max (p_1^2,  10 p_2^2)$.
\smallskip

Next we use the arguments of the proof of the uniqueness of the state-constraint solutions to give a new and very simple proof of the comparison result established 
in \cite{IM} for a notion of limited flux junction solutions, which are ``parametrized'' by their values at $0$.  As in the rest of this paper we concentrate on the time-independent problem.
\smallskip

The notion of solution introduced in \cite{IM} requires the Hamiltonian's to be, in addition to coercive,  quasiconvex and the condition at the junction involves the
nondecreasing part of the Hamiltonians. To simplify the presentation, here we assume that each Hamiltonian $H_i$  is convex and has no flat parts. 
\smallskip

If $p^0_i=\text{argmin} H_i$, \cite{IM} uses  the auxiliary Hamiltonians $H_i^-(p_i,0):=H_1(p_1,0) \ \text {if } \ p_i \leq p^0_i$ and\\
$H_i^-(p_i,0):=H_i(p^0_i,0) \ \text {if } \ p_i \geq p^0_i$,  to define, for any $A\in \R$,  the $A$-flux limiter \\
$H_A(p):= \max (A, \underset {1\leq i \leq K} \max H^-_i (p_i, 0))$. 
\smallskip

The following definition was introduced in \cite{IM}.
\begin{definition} An $A$-flux limited sub (respectively super)-solution to junction problem is a viscosity sub(respectively super)-solution, for each $i$,  to $u + H_i(u_{x_i}, x_i) \  \text{in} \ I_i$ and $u+H_A(u_{x_1}, \ldots, u_{x_K})  \text{at } \  x=0.$
\end{definition} 
We remark that, in addition to the severe restriction of convexity, the $A$-flux limited solutions are classified essentially by their values at the origin and not the Kirchoff-type Neumann solution we use here, which is more natural for the interpretation of the solution.
\smallskip

Motivated by the control theoretic interpretation of the problem \cite{IM}  constructed a rather elaborate  test function to deal 
with the case that points coming up in the uniqueness proof are at the origin. 
\smallskip

Here we present a rather simple proof for this uniqueness. To simplify the arguments we consider continuous solutions and prove the following.
\begin{proposition}\label{pl11} Let $u,v$ be continuous $A$-flux limited sub- and super-solutions respectively. Then $u\leq v$ on ${\bar I}$. 
\end{proposition}
\begin{proof}  The first observation is that $u(0)\leq -A$. Indeed, for $\ep>0$ small, consider a test function $\phi \in C^1(I) \cap C^{0,1}(\bar I)$ such that  $\phi_i(x_i)=- x_i/\ep.$ It is easy to see that $u-\phi$ attains a local maximum in a neighborhood of $0$ at some point ${\bar X}:=({\bar x}_1,\dots, {\bar x}_K)$. If $\bar X \in I_i$ for some $i$, then  $u(\bar X) + H_i(-1/{\ep}, \bar X) \leq 0$, which is  not possible if $\ep$ is sufficiently small since $H_i$ is coercive. Hence $\bar X=0$ and the definition yields  $u(0) +A \leq u(0) + H_A(D\phi(0),0) \leq 0.$
\smallskip

For the comparison we follow the proof or Theorem \ref{pl1} and recall that we only need to consider the case that the maximum of the ``doubled'' function is achieved for all $i$'s at  some $(\bar x_i, 0)$ with $\bar x_i <0.$ 
\smallskip

The definition of the $A$-flux limited super-solution then yields
$v(0) + H_A(\frac{{\bar x}_1 +\delta}{\ep}, \ldots, \frac{{\bar x}_K +\delta}{\ep},0)\geq 0.$
\smallskip

If  $H_A(\frac{{\bar x}_1 +\delta}{\ep}, \ldots, \frac{{\bar x}_K +\delta}{\ep},0)=A$, then 
$v(0) + A \geq 0$, that is  $v(0) \geq -A\geq u(0)$, and we may conclude. 
\smallskip
  
If \ $H_A(\frac{{\bar x}_1 +\delta}{\ep}, \ldots, \frac{{\bar x}_K +\delta}{\ep},0)=\underset{1\leq i \leq K} \max H_i^-(\frac{{\bar x}_1 +\delta}{\ep}, \ldots, \frac{{\bar x}_K+\delta}{\ep},0)$, then \\

$v(0) + \underset{1\leq i \leq d} \max H_i(\frac{{\bar x}_1 +\delta}{\ep}, \ldots, \frac{{\bar x}_d +\delta}{\ep},0) \geq  v(0) + \underset{1\leq i \leq d} \max H^-_i(\frac{{\bar x}_1 +\delta}{\ep}, \ldots, \frac{{\bar x}_K +\delta}{\ep},0) \geq 0$,

and  we may conclude as in the proof of Theorem \ref{pl1}.

\hskip6.5in  $\blacksquare$
\end{proof}

We conclude with a proposition, which we state without a proof, which provides information about the location of the possible elements of the superdifferential at the junction of a sub-solution in $I$. An immediate consequence is that in the quasi-convex studied in \cite{IM}, there is no need to use in advance the decreasing parts of the Hamiltonians in order to define the flux-limited solution at the junction.

\begin{proposition}\label{pl6.1}  Assume that $u \in C(\bar I)$ solves $u+H(u_x,x)\leq 0$ in $I$. Then either $u$ is the state-constraint solution in $\bar I$ or $\underset{x\to 0^-} \limsup \frac{u(x)-u(0)}{x} \leq \bar P$,  where \\
$\bar P:=\inf\{z\in \R: H(z,0) \leq H(p,0) \ \text{for all} \ z\leq p \}.$  
\end{proposition}
\section{Extensions }
\label{sec:ext}

A first extension of our results is about time dependent junction problems.

\begin{definition}\label{takis0} (i) $u \in C({\bar I} \times [0,T]; \R)$  is a state-constraint sub-solution to the junction problem if  
\begin{equation}\label{takis3.1}
u_{i,t}+H_i(u_{x_i}, x_i)\leq 0 \ \text{in} \ I_i\times (0,T] \ \text{ for each $i$.} 
\end{equation}
(ii)~ $u \in C({\bar I}\times [0,T]; \R)$ is a state-constraint super-solution if 
\begin{equation}\label{takis3}
\begin{cases}
 u_{i,t}+H_i(u_{x_i}, x_i)\geq 0 \ \text{in} \ I_i, \times (0,T] \  \text{for each $i$, and} \\[1mm] 
 \underset {1\leq i \leq K}\max (u_{i,t}  + H_i(u_{x_i},0)) \geq 0.
\end{cases}
\end{equation}
(iii)~$u \in C({\bar I} \times [0,T]; \R)$ is a solution if it is both sub-and super-solution.
\end{definition}
As for the time independent problems discussed earlier the super-solution inequality at the junction is interpreted in the viscosity sense, that is  if, for $\phi \in C^1(I \times (0,T])\cap C^{0,1}(\bar I \times [0,T])$, $u-\phi$ has a (local) minimum at $(0,t_0)$ with $t_0\in (0,T]$, then $ \underset {1\leq i \leq K}\max \left[\phi_{i,t} (0,t_0)+ H_i(\phi_{x_i}(0,t_0),0)\right] \geq 0.$
\smallskip

%

The uniqueness of solutions as well the simple proof of the uniqueness of flux-limited solutions to the time dependentent  junction problem follow after some easy modifications of the arguments presented in the previous sections. The convergence of the Kirchoff second-order approximations require some additional arguments. The details are given in 
\cite{LS1}. 
\smallskip

Other possible generalizations to the so-called ``stratified'' problems were discussed by the first author in \cite{L} and will be also presented in \cite{LS1}.  
\smallskip

The following example is a typical problem.  Consider the  domain $\Sigma:=\Sigma_1 \cup \Sigma_2$ with $\Sigma_1:=(-\infty,0) \times \R \times \{0\}$ and  $\Sigma_2:=\{0\} \times \{0\} \times (-\infty,0)$ and the coercive nonlinearities $F$ and $H$.  The equation is:
\begin{equation*}
\begin{cases}
F(u_z,z) + u=0 \ \text{in } \  \Sigma_2, \\[1mm]
H(u_x,u_y, x, y) + u=0 \  \text{in} \ \Sigma_1,\\[1mm]
H(u_x,u_y, x, y) + u\geq 0 \  \text{on} \  \partial \Sigma_1,\\[1mm]
\min(H(u_x,u_y, x, y) + u, F(u_z,z) + u) \geq 0 \ \text{at} \  \{0\} \times  \{0\} \times  \{0\}.
\end{cases}
\end{equation*}
A more general multi-dimensional example, always for coercive nonlinearities, in the domain \ $\Sigma:=\{(x,y) \in \R^{K+d}: x_i\leq 0\}$ is
\begin{equation*}
\begin{cases}
H_i(u_{x_i}, D_y u, x_i,y) + u_i= 0 \ \text{in} \ (-\infty,0) \times \R^d, \\[1mm]
\underset{1\leq i \leq K}\max H_i(u_{x_i},  D_y u, 0,y) + u \geq 0 \ \text{in} \  \{0\} \times \R^d.
\end{cases}
\end{equation*}


\noindent ($^{1}$) College de France,
11 Place Marcelin Berthelot, 75005 Paris, 
and  
CEREMADE, 
Universit\'e de Paris-Dauphine,
Place du Mar\'echal de Lattre de Tassigny,
75016 Paris, FRANCE\\ 
email: lions@ceremade.dauphine.fr
\\ \\
\noindent ($^{2}$) Department of Mathematics 
University of Chicago, 
5734 S. University Ave.,
Chicago, IL 60637, USA\\ 
email: souganidis@math.uchicago.edu
\\ \\
($^{3}$)  Partially supported by the National Science Foundation.

\end{document}